\documentclass[11pt]{amsart}
\usepackage{amsmath, amssymb, latexsym}
\usepackage{pdfsync}
\usepackage{graphicx}
\usepackage{amsfonts}
\usepackage{enumerate}

\usepackage{amsthm}

\usepackage{color, accents,fullpage}

\newtheorem{theorem}{Theorem}

\newtheorem{cor}[theorem]{Corollary}

\newtheorem{lemma}[theorem]{Lemma}
\newtheorem{prop}[theorem]{Proposition}

\newtheorem{example}[theorem]{Example}
\newtheorem{remark}[theorem]{Remark}

\def\R{{\mathbb R}}
\def\RR{{\mathbb R}}
\def\NN{{\mathbb N}}
\def\PP{{\mathbb P}}
\def\C{{\mathbb C}}

\def\P#1{{\mathbb P}^#1}

\newcommand{\inter}{\operatorname{int}}
\newcommand{\cl}{\operatorname{cl}}
\newcommand{\CC}{\mathbb{C}}
\newcommand{\rank}{\operatorname{rank}}

\numberwithin{theorem}{section}

\begin{document}

\author{Alessandra Bernardi, Grigoriy Blekherman, Giorgio Ottaviani}
\title{On real typical ranks}

\begin{abstract}
We study typical ranks with respect to a real variety $X$. Examples of such are tensor rank ($X$ is the Segre variety) and symmetric tensor rank ($X$ is the Veronese variety). We show that any rank between the minimal typical rank and the maximal typical rank is also typical. We investigate typical ranks of $n$-variate symmetric tensors of order $d$, or equivalently homogeneous polynomials of degree $d$
in $n$ variables, for small values of $n$ and $d$. We show that $4$ is the unique typical rank of real ternary cubics, and quaternary cubics have typical ranks $5$ and $6$ only. For ternary quartics we show that $6$ and $7$ are typical ranks and that all typical ranks are between $6$ and $8$. For ternary quintics we show that the typical ranks are between $7$ and $13$.

\end{abstract}
\maketitle

\section{Introduction}

Let $K$ be a field, let $X\subset {\mathbb P } (K^{N+1})$ be a projective variety and let $\hat{X}\subset K^{N+1}$ be the cone over $X$.
{ We assume that $X$ is nondegenerate, namely that $X$ is not contained in a hyperplane.}
The \emph{rank} of a point $x\in K^{n+1}$ with respect to $X$ ($X$-rank for short) is the minimum $k$ such that there exists a decomposition 
$x=\sum_{i=1}^k{\lambda_i}x_i$, with $x_i\in \hat{X}, {\lambda_i\in K}$. 

\indent We note that, either over $\mathbb{C}$ or $\mathbb{R}$, the rank with respect to the Segre variety is the usual tensor rank, and the rank with respect to the Veronese variety is the symmetric tensor rank and these are the examples of greatest relevance in applications.

If $K=\mathbb{C}$, then the set $S_k:=\{x\in{\mathbb P}^N\; |\; \textrm{rank}(x)=k\} \subset \mathbb{P}^N$ is a constructible set and $S_k$ has non-empty interior for a unique value of $k$, which is called the \emph{generic rank}. This holds in both Euclidean and Zariski topologies. If $K=\R$ then
$S_k$ is a semialgebraic set and it has non-empty interior (with the Euclidean topology) for several values of $k$, which are called the \emph{typical ranks}.
It is known that the smallest typical (real) rank coincides the generic (complex) rank \cite{BT}.

Typical ranks are of interest for applications, since on an open subset the rank of a point does not change under small perturbations in the data  (see eg. \cite{ComoT08:lasvegas}, \cite{ComoTDC09:laa}). For a real projective variety $X$ with dense real points, let $X_{\mathbb{C}}$ denote the complexification of $X$, i.e. $X_{\CC}$ consists of all complex points defined by the vanishing ideal of $X$. It is known that when $X$ is irreducible, the smallest typical $X$-rank over $\R$ coincides with
the generic $X_{\CC}$-rank over $\C$ \cite{BT}. 

Note that tensors of order $2$ are simply matrices and tensor rank coincides with the usual matrix rank over any field $K$. In this case there is a unique typical rank over $\RR$. Similarly, symmetric tensors of order $2$ are symmetric matrices, and again the tensor rank coincides with the usual matrix rank over any field $K$. When $K=\RR$ there is unique typical rank for tensors in $S^2\RR^n$, which is $n$. 
The situation is much more complicated when the order of the tensors is larger than $2$.
It was observed in \cite{MR1088949} that the typical ranks for tensors in ${\mathbb P}(\R^2\otimes\R^2\otimes\R^2)$
are $2$ and $3$. More sparse results can be found in \cite{MR1693919}, \cite{MR1818596}, \cite{MR2257591}, \cite{ComoTDC09:laa}, \cite{SSM}.
For monomials, results which show the difference between the real and the complex case can be found in \cite{BCG,CCG}.

For symmetric tensors, it was shown in \cite{CO} that for bivariate symmetric tensors of order $d$, the top typical rank is $d$ and the lowest typical rank is $\left\lfloor \frac{d}{2}\right\rfloor+1$. They also showed that for $d=5$, there are three typical rank $3,4,5$ and they conjectured that for all $d$ all ranks between $\left\lfloor \frac{d}{2}\right\rfloor+1$ and $d$ are typical. This conjecture was proved in \cite{B}, see also \cite{CR}. Our first main result is a considerable extension of \cite{B}: we show that for any real projective variety $X$ any rank between the lowest typical rank and the top typical rank is also typical:
\begin{theorem}\label{thm:typical}
Let $X\subset \PP_{\RR}^N$ be a real projective variety. Then any $X$-rank between the lowest typical rank and the highest typical rank is also typical.
\end{theorem}

Theorem \ref{thm:typical} is proved in Section \ref{sec:typical}. In addition to \cite{B} quoted above, we mention that the case of Theorem \ref{thm:typical} when $X$ is a Segre variety with three factors was proved
in \cite[Theorem 7.1]{Fr}. In the remainder of the paper we investigate typical ranks of $n$-variate real symmetric tensors of order $d$ (equivalently real forms in $n$ variables of degree $d$) for small values of $n$ and $d$. 
{The case of ternary cubics is classical and goes back essentially to XIX century.
This case has been recently considered in detail in \cite{Ba}, where the complete classification, according to rank, is shown.} 

We summarize our results in the following Theorem:

\begin{theorem}\label{main} \hspace{0.2cm} 
Let $S^d\mathbb{R}^{n}$ be the space of $n$-variate real symmetric tensors of order $d$.
\begin{enumerate}
\item{}\label{(1)}Tensors in $\mathbb{P}(S^3\mathbb{R}^{3})$ have one typical rank, which is $4$.
\item{}\label{(2)}Tensors in $\mathbb{P}(S^3\mathbb{R}^{4})$ have two typical ranks, which are $5$ and $6$.
\item{}\label{(3)}Tensors in $\mathbb{P}(S^4\mathbb{R}^{3})$ have at least two typical ranks, which are $6$, $7$.
The maximum typical rank is at most $8$.
\item{}\label{(4)}Tensors in $\mathbb{P}(S^5\mathbb{R}^{3})$ have {at least two typical ranks, which are $7$ and $8$.
The maximum typical rank is  at most $13$.}
\end{enumerate}
\end{theorem}

(\ref{(1)}) and (\ref{(2)}) are proved in Section \ref{v3P2}. (\ref{(3)}) and (\ref{(4)}) are proved respectively in Section \ref{quartics} and \ref{quintics}.
Section \ref{sec:symeven} is devoted to the special class of symmetric even quartics.
We leave as an open problem to complete the classification of typical ranks in $S^4\mathbb{R}^{3}$. We remark that the maximum complex rank
in $S^4\CC^{3}$ is $7$, see \cite{DeP}.

\section{Typical Ranks}\label{sec:typical}

Let $X \subset \PP^n_{\RR}$ be a real projective variety and let $\hat{X} \subset \RR^{n+1}$ be the cone over $X$.  We will use $\inter S$ and $\accentset{\circ}{S}$ to denote the interior and $\cl S$ and $\bar{S}$ to denote the closure of a set $S \subseteq \RR^k$
 in the Euclidean topology. We first prove an elementary Lemma about real semialgebraic sets.

\begin{lemma}\label{lem:intcl}
Let $S$ be a semialgebraic set in $\RR^k$. Then $\inter (\cl S) \subseteq \cl (\inter S)$.
\end{lemma}
\begin{proof}
Since $S$ is semialgebraic we can write $S$ as a finite union of non-empty semialgebraic sets $H_i$: $S=\cup_{i=1}^m H_i$ where each $H_i$ is of the form:
$$H_i=\{x \in \RR^k \,\,\mid \,\, p_j(x)>0, \,\, h_\ell(x)=0\}.$$
Define $S'$ as the union of the sets $H_i$ where only inequality constraints are used. It follows that 
$S'$ is open and $\inter (\cl S')=\inter (\cl S)$, and $\cl (\inter S')=\cl(\inter S)$. Therefore it suffice to prove the Lemma for a basic open semialgebraic set $S'$. However in this case we have
$$\inter (\cl S')\subseteq \cl S'= \cl (\inter S').$$
\end{proof}

\begin{theorem}
Let $S_{\leq r}$ be the set of points of $X$-rank at most $r$ in $\RR^{n+1}$. If $r$ is a typical rank, but $r+1$ is not typical, then $\cl \accentset{\circ}{S}_{\leq r}=\RR^{n+1}$.
\end{theorem}
\begin{proof}
 Since $r+1$ is not a typical rank it follows that $\accentset{\circ}{S}_{\leq r+1} \subseteq \bar{S}_{\leq r}$, otherwise there exists a point of rank $r+1$ in the interior of ${S}_{\leq r+1}$ which is not in the closure of $ \bar{S}_{\leq r}$, which certifies that $r+1$ is a typical rank.

 Therefore $\accentset{\circ}{S}_{\leq r+1} \subseteq \inter \bar{S}_{\leq r}$.  Since $r$ is a typical rank we see that $\accentset{\circ}{S}_{\leq r}$ is non-empty. Let $A=\cl \accentset{\circ}{S}_{\leq r}$. The set $S_{\leq r}$ is semialgebraic and by Lemma 
\ref{lem:intcl} we know that $$\accentset{\circ}{S}_{\leq r+1} \subseteq A.$$

Now consider $p \in A$. We know that there exist $p_i \in \accentset{\circ}{S}_{\leq r}$ such that $p_i$ converge to $p$. For all $x\in \hat{X}$ and for all $i$ we know that $p_i+ x \in \accentset{\circ}{S}_{\leq r+1}$. Therefore it follows that $p+x \in A$ for all $x\in \hat{X}$. Then proceeding by induction for all $k \in \NN$, and for all $x_i \in \hat{X}$ we have $$p+\sum_{i=1}^k x_i \in A.$$ But then it follows that $A=\RR^{n+1}$ as desired.
\end{proof}

Theorem \ref{thm:typical} now follows immediately. We recall that the generic rank for symmetric tensors in $\mathbb{P}(S^d\mathbb{C}^{n+1})$
has been computed by Alexander-Hirschowitz in  \cite{AH}. It is equal to $n+1$ for $d=2$ and
to $\lceil\frac{{{n+d}\choose d}}{n+1}\rceil$ for $d\ge 3$, unless
$(d,n)=(4,2), (4, 3), (4,4), (3,4)$, when the previous bound needs to be increased by one. We call this value
the Alexander-Hirschowitz generic rank.
By combining Theorem \ref{thm:typical} with \cite[Theorem 2]{BT}, we get
\begin{cor}\label{CorGreg}
For real symmetric tensors, every rank between the Alexander-Hirschowitz (complex) generic rank and the top typical real rank is also typical.
\end{cor}

\section{Ternary and quaternary cubics}\label{v3P2}

In this section we compute the typical ranks for points $x\in \mathbb{P}(S^3\mathbb{R}^{n+1})$ with respect to the Veronese variety $X=v_3(\P n_{\mathbb{R}})$
for $n=2$ and $3$. The following result is classical and goes back essentially to a result by De Paolis published in 1886,
see \cite[libro III, \S I.9]{MR966664}. It has been reviewed recently by Banchi in \cite{Ba}.

\begin{theorem}\label{thm:typcubic}[De Paolis]
For real plane cubics there is only one typical rank, which is $4$.
\end{theorem}

\begin{proof} The general plane cubic is $SL(3)$-equivalent to
 Hasse form $f_{\lambda}=x^3+y^3+z^3+6\lambda xyz$, for some $\lambda\in\RR$.

For $\lambda\neq -\frac 12$ we have the rank $4$ decomposition

$$f_{\lambda}=c_0(x+y+z)^3+c_1((1+\lambda) x-\lambda y-\lambda z)^3+c_2(-\lambda x+(1+\lambda) y-\lambda z)^3+c_3(-\lambda x-\lambda y+(1+\lambda)z)^3$$

where $c_0=\frac{\lambda(\lambda^2+\lambda+1)}{(2\lambda+1)^2}$ and  $c_i=\frac{1}{(2\lambda+1)^2}$ for $i=1,\ldots ,3$. For details on how this decomposition can be found we refer to \cite{Ba}.
\end{proof}

Let $V=\R^4$ and consider a general $\phi\in S^3V$. Over $\C$
the generic rank is $5$ and moreover Sylvester pentahedral theorem asserts that the decomposition is unique, \cite{MR966664}.

Consider indeed the catalecticant morphism
$$B_{\phi}\colon V^{\vee}\to S^2V$$
obtained by contraction with $\phi$.

Since the degree of subvariety of $S^2V$ given by rank $\le 2$ matrices is $10$, there are $10$ points $v_i$ such that $\textrm{rank}(B_{\phi}(v_i))\le 2$.
If $\phi=\sum_{i=1}^5l_i^3$ then the $10$ points $Q_{ijk}=\{l_i=l_j=l_k=0\}$
satisfy $\textrm{rank}(B_{\phi}(Q_{ijk}))\le 2$. This shows that the decomposition is unique and gives 
an algorithm to compute it.

Since $\phi=\overline\phi$, we get that there are three possibilities (up to reordering), all occurring for $\phi$ lying in a set of positive volume:
\begin{enumerate}[(i)]
\item\label{4.i} all $l_i$ are real, in this case all ten points $Q_{ijk}$ are real;

\item\label{4.ii} $l_1,l_2,l_3$ are real, $\overline l_5=l_4$, in this case exactly four among the points $Q_{ijk}$ are real;

\item\label{4.iii} $l_1$ is real, $\overline l_3=l_2$, $\overline l_5=l_4$, in this case exactly two among the points $Q_{ijk}$ are real.
\end{enumerate}
The following equality is straightforward:
\begin{lemma}\label{c2r3} 
$$(x+iy)^3+(x-iy)^3=-(x+y)^3-(x-y)^3+4x^3.$$
\end{lemma}

{The uniqueness shows immediately that in the cases (\ref{4.ii}) and (\ref{4.iii}) above the rank is $\ge 6$.}

\begin{prop}\label{prop:cubic}
In case (\ref{4.i}) the rank is $5$, which is typical for cubic surfaces in
$\mathbb{P}(S^3\mathbb{R}^{4})$.
In the cases (\ref{4.ii}) and (\ref{4.iii}) the rank is $6$, which is again typical.
\end{prop}

\begin{proof}
In the case (\ref{4.ii}),
by applying Lemma \ref{c2r3}, it follows that the rank is $6$, which is typical.
{If $z=(z_0,z_1,z_2)$, denote by $P_z(\phi)$ the polar
$\sum_{i=0}^2z_i\frac{\partial \phi}{\partial z_i}$.}
In the case (\ref{4.iii}), let {$z=Q_{123}$ be one of the two points $Q_{ijk}$ which are real
(the other one is $Q_{145}$).  }
Then $P_z(\phi)=(l_4)^2+(l_5)^2$, after multiplying $l_4$ and $l_5$ by a suitable scalar.
Let $l_4=a+ib$ and $l_5=a-ib$ be the real decompositions.
Then $P_z(\phi)=a^2-b^2$. Let $\psi=\left(3P_z(a)\right)^{-1}a^3-\left(3P_z(b)\right)^{-1}b^3$,
then $P_z(\phi-\psi)=0$, it follows that $\phi-\psi$ can be expressed as a cubic form in just three variables,
hence it has border rank $\le 4$, by Theorem \ref{thm:typcubic}. In particular $\phi=\psi+(\phi-\psi)$ has border rank $\le 2+4=6$, 
and $6$ is again a typical rank. This concludes the proof. 
\end{proof}

\medskip
(1) of Theorem \ref{main} follows from Theorem \ref{thm:typcubic}. 
(2) of Theorem \ref{main} follows from Proposition \ref{prop:cubic}.

\section{Ternary Quartics}\label{quartics}

In this section we estimate the typical ranks for points $x\in \mathbb{P}(S^4\mathbb{R}^3)$ with respect to the Veronese variety $X=v_4(\P 2_{\mathbb{R}})$.

\begin{prop}\label{prop:typquartic} Typical ranks for ternary quartics are between $6$ and $8$.
\end{prop}
\begin{proof}
A generic complex ternary quartic has rank $6$, therefore $6$ is the smallest typical rank for real ternary quartics,
by \cite[Theorem2]{BT}.
For a generic $f\in S^4\mathbb{R}^3$ and a generic $Q\in\PP^2$ we have that  $P_Q(f)\in S^3\mathbb{R}^3$ and  $rk(P_Q(f))=4$ (by Theorem \ref{thm:typcubic}),
say $P_Q(f)=\sum_{i=1}^4l_i^3$. If $Q\notin l_i$,  then there exists $g\in S^4V$  
 such that 
$$f=\sum_{i=1}^4\left(4P_Q(l_i)\right)^{-1} l_i^4+g$$ and $P_Q(g)=0$.
Since  $g$ is a binary form  (it may be written as a polynomial in only two variables) and since for $S^3\RR^2$ we have two typical ranks $r=3,4$ (see \cite{CO,B}),
then the typical ranks of ternary quartics are $\le 4+4=8$.
\end{proof}

\begin{remark}
Let $f\in S^4\mathbb{R}^3$. If there exists $Q$ such that $P_Q(f)=\sum_{i=1}^3l_i^3$ has real rank three (so it is $SL(3)$-equivalent to a real Fermat cubic curve),
then the argument of the proof of Proposition \ref{prop:typquartic} shows that $f$ has rank $\le 7$.
Hence for a large part of $S^4\mathbb{R}^3$ the rank is $\le 7$. The following example shows that this argument cannot be applied
to a dense subset of ternary quartics.
\end{remark}

\begin{example} Let
$f=x^4+y^4+z^4-5(x^2y^2+x^2z^2+y^2z^2)$ (its corresponding real plane curve has three ovals).
The equation of the hypersurface which is the Zariski closure of rank three cubics is a classical invariant $S$ of degree $4$,
called the Aronhold invariant, see \cite{O} for a pfaffian formula that makes it easy to compute. We claim that there is no $Q\in \PP^2$ 
such that $S(P_Q(f))=0$ where $S$ is the Aronhold invariant, and this remains true for every small perturbation of $f$.
Indeed $S(f)=\{Q|S(P_Q(f))=0\}$ is again a quartic curve in $Q=(x,y,z)$, with equation  
$25(x^4+y^4+z^4)+71(x^2y^2+x^2z^2+y^2z^2)=0$, which has no real points and lies in the interior of nonnegative quartics. $S(f)$ is equipped with a theta characteristic $\theta$ and $f$ was called the Scorza quartic of $(S(f),\theta)$ in \cite{DoKa}. 
\end{example}

We will present three different ways of showing that $7$ is a typical rank for ternary quartics. The first one is based on the topology of the real variety of the ternary quartic, the second one is based on the signature of the middle catalecticant, and the third one is based on the variety of decompositions of a generic complex ternary quartic as a sum of six $4$-th powers.

\subsection{Topology, Middle Catalecticant and Real Rank}

We observe some interplay between the topology of the real curve $V_\RR(f)$ defined by the ternary quartic $f$, the signature of its middle catalecticant $M_f$ and the real rank of $f$.

\begin{prop}\label{prop:line}
Let $f$ be a ternary quartic such that there exists a real line $L$ intersecting $V_\RR(f)$ in $4$ real points counting multiplicity, at least $2$ of which are distinct. Then any decomposition of $f$ as a sum of $4$-th powers must include $2$ powers with a positive sign and $2$ powers with a negative sign.

\end{prop}

\begin{proof}
It suffices to show the claim for the number of positive signs, since the proof for the number of negative signs is identical. Consider the restriction $\hat{f}$ of $f$ to $L$. Then $\hat{f}$ is a bivariate form with all real roots, which is not a $4$-th power. It is shown in \cite[Corollary 2.6]{Rez1} that any decomposition of $\hat{f}$ as a sum of $4$-th powers must include two powers with a positive sign. It follows that the same is true for $f$.
\end{proof}

We now prove our first restriction on the rank of $f$ based on the geometry of $V_\RR(f)$ and the signature of the middle catalecticant $M_f$.

\begin{cor}\label{cor:line}
Let $f$ be a ternary quartic such that the signature of $M_f$ is $(5,1)$ and there exists a real line $L$ intersecting $V_\RR(f)$ in $4$ real points counting multiplicity, at least $2$ of which are distinct. Then $\mathrm{rk}_{\RR}(f) \geq 7$.
\end{cor}

\begin{proof}
Since the signature of $M_f$ is $(5,1)$ we know that any decomposition of $f$ as a linear combination of $4$-th powers must include at least $5$ powers with positive signs and from Proposition \ref{prop:line} we see that we must also have at least $2$ powers with negative signs. Thus $\mathrm{rk}_{\RR}(f) \geq 7$.
\end{proof}

It now quickly follows that $7$ is a typical rank for ternary quartics.

\begin{example}
Let $$f=z^4+2(x^2y^2+x^2z^2+y^2z^2)-\frac{1}{10}x^4-\frac{1}{10}y^4.$$
It is easy to calculate that the middle catalecticant $M_f$ has signature $(5,1)$. See Proposition \ref{prop:typ7} for details. It is also easy to check that the line $y=0$ intersects $V_\RR(f)$ in $4$ distinct real points. It follows that any sufficiently small perturbation of $f$ will have signature $(5,1)$ and will intersect the line $y=0$ in $4$ real points. Therefore we see that $7$ is a typical rank for ternary quartics.

\end{example}

A ternary form is called \textit{hyperbolic} if there exists a point $p \in \RR\PP^2$ such that every line through $p$ intersects $V_\RR(f)$ in $\deg f$ many points (counted with multiplicity).

\begin{prop}
Let $f$ be a hyperbolic ternary quartic, which is not a power of a linear form. Any decomposition of $f$ as a sum of $4$-th powers must include $3$ powers with a positive sign and $3$ powers with a negative sign.
\end{prop}

\begin{proof}
Again, it suffices to show the claim for the number of positive signs. Suppose not, and let $$f=\ell_1^4+\ell_2^4-\sum_{i=1}^k m_i^4$$
be a decomposition of $f$. Let $p\in \RR\PP^2$ be a point such that $f$ is hyperbolic with respect to $p$. The lines defined by vanishing of $\ell_1$ and $\ell_2$ meet in a unique point $q$. Let $L$ be a line through $p$ and $q$ and consider the restriction $\hat{f}$ of $f$ to $L$. From our construction of $L$ it follows that restrictions of $\ell_1$ and $\ell_2$ to $L$ are constant multiples of each other. Therefore we obtain a decomposition of $\hat{f}$ as a sum of $4$-th powers with only $1$ positive sign, which is again a contradiction by \cite[Corollary 2.6]{Rez1}.
\end{proof}

From the above proposition we immediately obtain the following corollary relating the rank of hyperbolic ternary quartics and the signature of middle catalecticant $M_f$. The proof is identical to the proof of Corollary \ref{cor:line}.

\begin{cor}\label{cor:hyp}
Let $f$ be a hyperbolic ternary quartic. If the signature of $M_f$ is $(4,2)$ then $\mathrm{rk}_{\RR}(f)\geq7$ and if the signature of $M_f$ is $(5,1)$ then $\mathrm{rk}_{\RR}(p)\geq8$.
\end{cor}

\begin{remark}
It is possible to show via a direct computation that the following quartic is hyperbolic, the signature of its middle catalecticant is $(3,3)$ and it has rank $6$:
$$f=2x^4+y^4+z^4-3x^2y^2-3x^2z^2-y^2z^2.$$
Below we sketch the computational proof  that the rank is $6$. Consider $p=f+c((x+my+2mz)^4+(x+2my+mz)^4)$ with $$c=\frac{7 (2144 + 113 \sqrt{394})}{3860} \hspace{.5cm} \text{and} \hspace{.5cm} m=-\frac{1}{3} \sqrt{\frac{  \sqrt{394}-13}{5}}.$$
It is possible to check that the rank of the middle catalecticant $M_p$ of $p$ is $4$ and furthermore the pair of quadrics in the kernel of $M_p$ intersect in $4$ real points. Therefore the rank of $f$ is at most $6$ and since $M_f$ is full rank it follows that the rank of $f$ is exactly $6$. Therefore we see that not all hyperbolic ternary quartics have rank at least $7$.
\end{remark}

We now demonstrate another example of a ternary quartic of typical rank $7$:

\begin{example}\label{4.10}
Let $$f=x^4 + y^4 + z^4 - 3x^2 y^2 - 3 x^2 z^2 + y^2 z^2.$$
It is easy to check that any line through $[1:0:0]$ intersects $V_{\RR}(f)$ in $4$ distinct real points; a picture of $V_{\RR}(f)$ in the plane $x=1$ is shown below. The signature of $M_f$ is $(4,2)$. It follows by Corollary \ref{cor:hyp} that $\mathrm{rank}_\RR(f) \geq 7$ and the same also holds for a small perturbation of $f$.
\end{example}

\begin{figure}[b]
\centering
\includegraphics[scale=0.8]{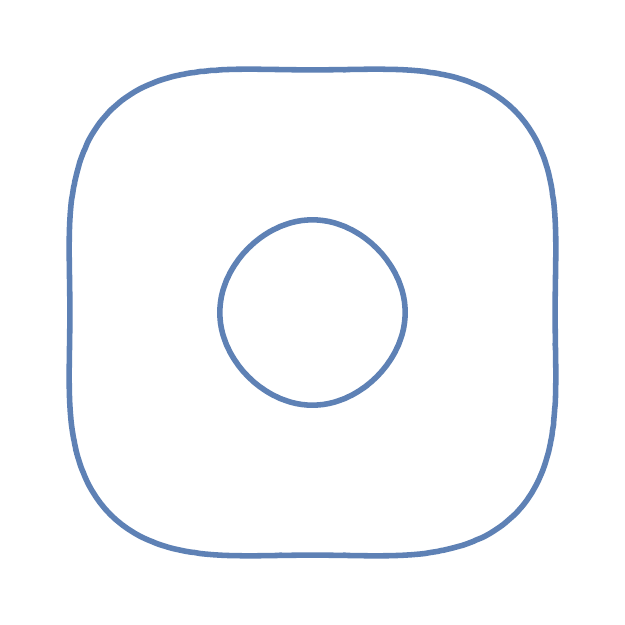}
\caption{The real variety of $f=x^4 + y^4 + z^4 - 3x^2 y^2 - 3 x^2 z^2 + y^2 z^2.$}
\end{figure}

Corollary \ref{cor:hyp} raises a possibility of showing that $8$ is a typical rank for ternary quartics by exhibiting a hyperbolic quartic $f$ such that $M_f$ has signature $(5,1)$. We have not been able to construct such an example, and we believe that hyperbolic quartics with signature $(5,1)$ do not exist.  In the example below we present some evidence toward nonexistence of such ternary quartics. 

\begin{example}
It is not difficult to show that if $f$ is the product of the following two conics
\begin{equation}\label{2conics}
f=(a^2x^2+b^2y^2-c^2z^2)(d^2x^2+e^2y^2-g^2z^2)
\end{equation}
with $a,b,c,d,e,g\in \mathbb{R}$,
then its middle catalecticant $M_f$ cannot have signature $(5,1)$.
It is easy to calculate that the characteristic polynomial of the middle catalecticant
has the following eigenvalues
$l_1=-4c^2e^2-4b^2g^2$, $l_2=-4c^2d^2-4a^2g^2$, $l_3=4b^2d^2+4a^2e^2$ and 
the three roots of a degree three polynomial in $l$ starting as follows: $p(l)=l^3+24l^2(-a^2d^2-b^2e^2-c^2g^2)+\cdots$.
 If we want that the signature of $M_f$ is $(5,1)$ then we need that all roots of $p$ are negative.
Now the coefficients of $l^2$ and $l^3$ in $p$ have different signs, so, since $p$ has to have 3 real solutions,  by Descartes rule of signs, there is at least one positive root of $p$. This shows that signature of $M_f$ cannot be $(5,1)$.
\end{example}

\begin{example} In the previous example, one can be more precise and see that only one root of $p$ that can be zero, the other two will always have opposite signs. Therefore the only two possible signatures for $M_f$ are $(4,2)$
(for example one can take $f=(x^2+y^2-z^2)(x^2+9y^2-z^2)$)
and $(3,3)$ (for example one can take $f=(x^2+y^2-z^2)(x^2+\frac{1}{100}y^2-z^2)$). Moreover this discussion also shows that if $f$ represents a quartic like (\ref{2conics}) with singular middle catalecticant $M_f$, then the signature of $M_f$ can only be $(3,2)$.
\end{example}

\subsection{Signature of the Middle Catalecticant and Real Rank}
We now show that $7$ is a typical rank for ternary quartics in a different way.
We need the following

\begin{theorem}[Reznick]\cite[Theorem 4.6]{Rez}\label{thm:rez}
Let $f$ be a real ternary quartic and let $M_f$ be its middle catalecticant.
If $M_f\succeq 0$, then $$\mathrm{rk}_{\RR}(f)=\mathrm{rk}(M_f),$$
where $\mathrm{rk}(M_f)$ is the usual matrix rank of $M_f$.
\end{theorem}

\begin{prop}\label{prop:typ7}
Let $$p=z^4+2(x^2y^2+x^2z^2+y^2z^2)-\frac{1}{10}x^4-\frac{1}{10}y^4$$
Then $\mathrm{rk}_{\RR}(p)=7$ and every small perturbation of $p$ has real rank $7$.
Hence $7$ is a typical rank for ternary quartics.
\end{prop}

\begin{proof} We write the middle
 catalecticant matrix $M_p$ with respect to the monomial basis $\{x^2,y^2,z^2,xy,xz,yz\}$. Then we have 

$$M_p=\left[\begin{array}{cccccc}\frac{-1}{10}&\frac{1}{3}&\frac{1}{3}&0&0&0\\ \frac{1}{3}&\frac{-1}{10}&\frac{1}{3}&0&0&0\\\frac{1}{3}&\frac{1}{3}&1&0&0&0\\0&0&0&\frac{1}{3}&0&0\\0&0&0&0&\frac{1}{3}&0\\0&0&0&0&0&\frac{1}{3}\end{array}\right]$$ \\

It is easy to see that $M_p$ has $5$ positive eigenvalues and $1$ negative eigenvalue. We first observe that $\bar{p}=p+13/30x^4+13/30y^4$ has a positive semidefinite catalecticant matrix of rank $5$ and therefore the rank of $p$ is at most $7$. Furthermore, for any $q$ in a small neighborhood of $p$ we can take $\bar{q}=q+ax^4+by^4$ and the catalecticant matrix of $\bar{q}$ will be diagonally dominant with positive diagonal entries, for sufficiently large $a$ and $b$. This implies that for any $q$ in the neighborhood of $p$ we can find $a,b \in \mathbb{R}$ such that $q+ax^4+by^4$ has a singular positive semidefinite catalecticant matrix and therefore the rank of $q$ is at most 7 by \ref{thm:rez}.

We now show that the rank of any form in a sufficiently small neighborhood of $p$ is at least $7$. Since $M_p$ has rank 6, we know that the rank of $p$ is at least $6$. If $p$ has rank $6$ then it must decompose as a linear combination of $6$ fourth powers, with 5 positive signs and $1$ negative sign. Rank $1$ catalecticants correspond to fourth powers of linear forms. With the monomial basis, for any fourth power $(ax+by+cz)^4$ the catalecticant matrix is $v^T v$ where $$v=\left(a^2,b^2,c^2,ab,ac,bc \right).$$

To argue that the rank of $p$ is at least $7$ we just need to show that $M_p+S$ is not positive semidefinite for any rank $1$ catalecticant matrix $S$. Now it suffices to look at the top $2 \times 2$ submatrix of $M_p+S$ which has the form:

$$\left[\begin{array}{cc}a^4-\frac{1}{10}&a^2b^2+\frac{1}{3}\\a^2b^2+\frac{1}{3}&b^4-\frac{1}{10}\end{array}\right].$$

The determinant of this submatrix is strictly negative for all $a,b \in \RR$, which means that it is never positive semidefinite, and therefore $M_p+S$ is not positive semidefinite for any rank one catalecticant matrix $S$. Thus $p$ has rank $7$. Moreover, the same argument applies to any $q$ in a sufficiently small neighborhood of $p$. Thus we have shown that $p$ has rank $7$, and any form in a neighborhood of $p$ has rank $7$ as well.
\end{proof}
\vskip 0.8cm

\subsection{Varieties of Sums of Powers and Real Rank}
We now sketch an alternative argument to prove that $7$ is a typical rank for real ternary quartics, suggested to us by F. Schreyer.
There is a natural birational correspondence, due to Mukai (\cite{M92}), between the moduli space of plane quartics defined over a field $K$ of characteristic 0, and the moduli space of prime Fano $3$-folds of genus  $12$. Let us briefly recall it.

The Fano $3$-folds of genus $12$ are classically called \emph{Fano 3-folds of type $V_{22}$} because of their anticanonical model of degree $22$ in $\mathbb{P}^{13}$.
\\
 Let $f\in S^4\CC^3$ be a homogeneous polynomial corresponding to a plane quartic $F\subset \mathbb{P}(\CC^3)$. A polar hexagon $\Gamma$ of $F$ is the union of six lines $\Gamma=\{l_1 \cdots l_6=0\}$ such that $f=l_1^4+ \cdots + l_6^4$. The hexagon $\Gamma$ can be identified with a point of $\mathrm{Hilb}^6(\mathbb{P}(\CC^3)^{\vee})$. The variety of sums of powers presenting $f$ is (see \cite{RS})
$$VSP(F,6):=\overline{\{\Gamma \in \mathrm{Hilb}^6(\mathbb{P}(\CC^3)^{\vee})\, | \, \Gamma \hbox{ is polar to } F\}}.$$
Mukai  (Theorem 11, \cite{M92}) showed that if $X$ is a prime Fano 3-fold of genus 12 over $K$, then there exists a plane quartic $F$ such that
$$X\simeq VPS(F,6).$$
Moreover he proved that the moduli space $\mathcal{M}_{Fano}$ of prime Fano 3-folds of genus 12 is birational to the moduli space $\mathcal{M}_q$ of plane quartics.
\\
Therefore, if there exist $l_1, \ldots , l_6\in S^1K^3$ such that $f\in S^4K^3$ can be written as $f=\sum_{i=1}^6l_i^4$, such a decomposition can be described as a point $\Gamma$ of the corresponding Fano 3-fold $X_f$ of type $V_{22}$.
Observe that a priori it is not sufficient that $X_f$ has a real point in order to claim that all the $l_i$'s involved in the decomposition of $f$ are reals. But vice-versa, if $l_1, \ldots , l_6\in S^1\mathbb{R}^3$ are such that $f=\sum_{i=1}^6\lambda_il_i^4$, then the corresponding point $\Gamma \in X_f$ is a real point.
Therefore if there exists an open set $\mathcal{U}\subset \mathcal{M}_{Fano}$ without real points, then the birational map $m:\mathcal{M}_{Fano}   \dasharrow
\mathcal{M}_{q}$ will produce an open set $m(\mathcal{U})\subset \mathcal{M}_{q}$ of plane quartics that cannot be written as sums of 6 four-powers of linear forms.  The existence of such an  $\mathcal{U}\subset \mathcal{M}_{Fano}$
is proved in  \cite[Theorem 1.1]{KO}. Hence the quartics in $m(\mathcal{U})$ have rank $> 6$ and by Theorem \ref{thm:typical} this is enough  to claim that $7$ is a typical rank for real plane quartics.

\section{Symmetric even quartics}\label{sec:symeven}

In this section we compute the rank of a special class of quartics. These are the quartics
$$f_{a}=a(x^4+y^4+z^4)+6(x^2y^2+x^2z^2+y^2z^2).$$
We present complete analysis of the case $a \geq -3$. We believe that in the case $a<-3$ the rank is always $6$. See Remark \ref{rem:-3} for more details.
Besides the trivial case $f_\infty=x^4+y^4+z^4$, which has rank $3$, $f_a$ are the only quartics which are symmetric (invariant by permuting variables)
and even (invariant by changing signs to any variable).
The discriminant can be computed by Sylvester formula (see \cite{LR}\S 2.2),
it is (up to scalars)
$$D(f_a)=a^3(a-3)^{14}(a+3)^6(a+6)^4.$$

The middle catalecticant of $f_{a}$, with respect to the monomial basis $\{x^2,y^2,z^2,xy,xz,yz\}$, is
(up to positive scalars)
$$M_{f_{a}}=\bgroup\begin{pmatrix}a&
      1&
      1&
      0&
      0&
      0\\
      1&
      a&
      1&
      0&
      0&
      0\\
      1&
      1&
      a&
      0&
      0&
      0\\
      0&
      0&
      0&
      1&
      0&
      0\\
      0&
      0&
      0&
      0&
      1&
      0\\
      0&
      0&
      0&
      0&
      0&
      1\\
      \end{pmatrix}\egroup,$$
and we may compute  $\det(M_{f_{a}})=(a-1)^2(a+2)$.
It is straightforward to prove the following two lemmas:

\begin{lemma}[Real shape]\label{lem:shape}
The quartic $f_{a}$ has the following real shape
\begin{itemize}
\item{} for $0\le a $  it is strictly positive (no real points), the complex curve $f_a$ is smooth, unless $f_3$ which is a double conic,
and $f_0$ which has three nodes.
\item{} for $-3< a< 0$ it consists of three ovals.
\item{} for $a=-3$ it splits into four lines.
\item{} for $-6< a <-3$ it consists of four ovals.
\item{} for $a = -6$ it consists of four points.
\item{} for $a < -6$ it is strictly negative (no real points).
\end{itemize}
\end{lemma}

\begin{lemma}[Signature]\label{lem:signature}
The middle catalecticant $M_{f_{a}}$ has the following signature
\begin{itemize}
\item{} for $1<a $ it is $(6,0)$.
\item{} for $a = 1$ it is $(4,0)$.
\item{} for $-2< a< 1$ it  is $(4,2)$.
\item{} for $a=-2$ it  is $(2,3)$.
\item{} for $-2< a$ it  is $(3,3)$.

\end{itemize}
\end{lemma}

Our main result in this section is the following
(in particular it proves that $6$ and $7$ are the only typical ranks in the class of even symmetric quartics).
\begin{theorem}\label{thm:evensymmetric}
The rank of $f_{a}$ is the following
\begin{itemize}
\item{} for $0<a $, $a\neq 1$, $\mathrm{rk}_{\mathbb R}(f_a)=6$.
\item{} for $a = 1$, $\mathrm{rk}_{\mathbb R}(f_a)=4$.
\item{} for $-3\le a\le 0 $, $\mathrm{rk}_{\mathbb R}(f_a)=7$.

\end{itemize}
\end{theorem}

\begin{remark}\label{rem:-3}
It seems to be possible to show that for $a<-3$ the rank of the corresponding quartic is always $6$. However our proof is highly computational and at present not suitable for inclusion in the article.
\end{remark}

\begin{remark}
By comparing Theorem \ref{thm:evensymmetric} with Lemma
\ref{lem:signature}, we see that there exist rank $7$ quartics with middle catalecticant having  the signature  $(4,2)$ or $(3,3)$.
Moreover, the example in Prop. \ref{prop:typ7} shows that there exist rank $7$ quartics with signature $(5,1)$.
In conclusion, there exist rank $7$ quartics with middle catalecticant having maximal rank and all possible indefinite signatures.
On the other hand, it is easy to produce rank $6$ quartics with middle catalecticant having maximal rank and all possible signatures,
just by adding general powers with appropriate signs.
\end{remark}

The easiest case $a=1$ follows by Lemma \ref{lem:signature} and by the identity
$$(x+y+z)^4+(x-y-z)^4+(-x+y-z)^4+(-x-y+z)^4=4f_1.$$ 

The proof of Theorem \ref{thm:evensymmetric} will be a consequence of the following Propositions
 \ref{prop:upbound}, \ref{prop:rkseven}, \ref{prop:rksix}, \ref{prop:invdisc}.

\begin{prop}\label{prop:upbound}
We have $\rank f_{a} \leq 7$.
\end{prop} 
\begin{proof}
Let $g=f_{a}+(1-a)(x^4+y^4+z^4)$. We observe that the catalecticant matrix $M_g$ of $g$ is positive semidefinite of rank $4$ and therefore $\rank g=4$ 
(by Theorem \ref{thm:rez}) and $\rank f_{a} \leq 7$.
\end{proof}

\begin{prop}\label{prop:rkseven}
If $-2<a\le 0$ then $\rank f_a = 7$.
\end{prop}
\begin{proof}
By Lemma \ref{lem:signature} the signature of $M_{f_{a}}$ is $(4,2)$, hence
 we know that $\rank f_{a} \geq 6$. Suppose that $\rank f_{a}=6$. Then $f_{a}$ is a linear combination of $6$ fourth powers of linear forms, with $4$ 
fourth powers having positive coefficients and $2$ fourth powers with negative coefficients.

For $i=1,2$ let $\ell_i$ be a real linear form: $\ell_i=\alpha_ix+\beta_iy+\gamma_iz$. To establish that $\rank f_{a} \geq 7$ it suffices to show that 
for any two linear forms $\ell_i$ the catalecticant matrix of $f_{a}+\ell_1^4+\ell_2^4$ is not positive semidefinite. We will show that this holds for $a=0$,
 which then establishes the claim for all $a\le 0$.

Let $C$ be the catalecticant matrix of $\ell_1^4+\ell_2^4$. Since $M_{0}$ is a matrix of full rank, it suffices to show that $M_{0}$ 
is not positive definite on the kernel of $C$.
The kernel of $C$ consists of all quadratic forms vanishing on the points $v_1=(\alpha_1,\beta_1,\gamma_1)$ and $v_2=(\alpha_2,\beta_2,\gamma_2)$. 
Let $L=rx+sy+tz$ be the unique line through $v_1$ and $v_2$ and let $V$ be the subspace of $\ker C$ consisting of forms $xL,yL,zL$. 
Restricting the quadratic form $M_{0}$ to the subspace $V$ we get the following matrix with respect to the basis $xL,yL,zL$:
$$M_{0}'=\begin{pmatrix} s^2+t^2&2rs&2st\\2rs&r^2+t^2&2st\\2rt&2st&r^2+s^2\end{pmatrix}.$$
Without loss of generality we may assume that $r\geq s\geq t$. The $(1,2)$ principal minor of $M_{0}'$ is $-3r^2s^2+t^4+s^2t^2+r^2t^2$, 
which is nonpositive by our assumption on $r,s,t$. Therefore $M_{0}$ is not positive definite on the kernel of $C$, which completes the proof.
\end{proof}

\begin{prop}\label{prop:rksix}
Suppose that $a>0$. Then $\rank f_{a} \leq 6$.
\end{prop}
\begin{proof}
For $a\geq 1$ the catalecticant matrix $M_{a}$ is positive semidefinite and the Proposition follows from
Theorem \ref{thm:rez}.
So we may assume $0< a< 1$. 

Let $$\omega=\frac{\sqrt{6}-\sqrt{2}}{2}, \,\, \,\,\,\,\,\, \ell_1=\omega x+z \,\,\,\,\,\, \text{and} \,\,\,\,\,\, 
c_1=\frac{(1-a)(2+a)}{(2-\sqrt{3})a(a+5)}.$$

We note that $c_1>0$ for $0<a <1 $. Let 

$$\omega_2=\frac{-\sqrt{6}-\sqrt{2}}{2},\,\,\,\,\,\, \,\,\,\,\,\, \omega_3=\sqrt{\frac{(3+\sqrt{3})a(1-a)}{6}} \,\,\,\,\,\,\,\,\, \text{and} \,\,\,\,\,\,\,\,\, \ell_2=\omega_2x+\omega_3y+z.$$
Finally, let $$c_2=\frac{6(1-a)}{(\sqrt{3}+2)a(3-a)(1+6a-a^2)}.$$
We note that for $0<a<1$, $\omega_3$ is a well-defined real number and $c_2>0$. Let $$g=f_a+c_1\ell_1^4+c_2\ell_2^4.$$ 
It is possible to show via direct computation that the catalecticant matrix $M_g$ of $g$ has rank $4$ for all $0<a<1$. Then $M_g$ is positive semidefinite, since the matrix $M_a$ has signature $(4,2)$. Therefore $g$ has rank $4$ by Theorem \ref{thm:rez} and $f_a$ has rank $6$.
\end{proof}

\subsection{Discriminant Pullback Hypersurface}  Let $f \in \RR[x,y,z]_4$. Let $\ell=\alpha x+\beta y+ \gamma z$. Consider the determinant of the middle catalecticant of $f+\ell^4$: $\det M_{f+\ell^4}$. Since the middle catalecticant of $\ell^4$ has rank $1$ we have:  
$$\det M_{f+\ell^4}=\det M_f+\tilde{f},$$
where $\tilde{f}$ is a ternary quartic in $\alpha,\beta$ and $\gamma$. Therefore we may define a map
$$\pi: \RR[x,y,z]_4 \rightarrow \RR[\alpha,\beta,\gamma]_4, \hspace{5mm} \pi(f)=\tilde{f}.$$
Note that coefficients of $\tilde{f}$ are quintic forms in the coefficients of $f$, coming from maximal minors of $M_f$. Therefore we have $\pi(-f)=-\pi(f)$.

\begin{prop}\label{prop:invdisc}
Suppose $-3\le a< -1$. Then  $$\rank f_{a} = 7.$$ 
\end{prop}
\begin{proof}
We compute $$\pi(f_a)=(a^2-1)(\alpha^4+\beta^4+\gamma^4)+(a-1)(a^2+a-4)(\alpha^2\beta^2+\alpha^2\gamma^2+\beta^2\gamma^2).$$
We want to show that, under our assumptions, $\pi(f_a)$ always has constant sign. Indeed, up to a scalar multiple, $\pi(f_a)=f_\alpha$ with $\alpha=\frac{6(a+1)}{a^2+a-4}$, unless $a=\frac{-1-\sqrt{17}}{2}$,
and in this last case $\pi(f_a)$ has constant sign. Moreover, if $-3\le a <  \frac{-1-\sqrt{17}}{2}$ then $\alpha<0$ and $\pi(f_a)$ has constant sign
by Lemma \ref{lem:shape}.
Last, if $ \frac{-1-\sqrt{17}}{2}< a <-1$ then $\alpha>0$ and $\pi(f_a)$ has constant sign by Lemma \ref{lem:shape}.

Suppose that $\rank f_a=6$ and $f_a=\sum c_im_i^4$, with $c_i=\pm 1$ and $m_i=\alpha_ix+\beta_iy+\gamma_iz$. 
Since the signature is $(3,3)$, and the signs of $c_i$ agree with the signature of $f_a$,
we may assume $c_1=-1$ and $c_2=1$. Furthermore $f_a-c_im_i^4$ has middle catalecticant of rank $5$ for all $i$. 
It follows that $\det(M_{f_a-c_im_i^4})=0$. Therefore we have $$\det(M_{f_a+m_1^4})=0 \hspace{.5cm} \text{and} \hspace{.5cm} \tilde{f_a}(\alpha_1,\beta_1,\gamma_1)=-\det M_{f_a}.$$
On the other hand we also have $$\det(M_{f_a-m_2^4})=\det(M_{-f_a+m_2^4})=0 \hspace{.5cm} \text{and} \hspace{.5cm} \widetilde{(-f_a)}(\alpha_2,\beta_2,\gamma_2)=-\tilde{f_a}(\alpha_2,\beta_2,\gamma_2)=-\det M_{f_a}.$$
Therefore we see that $\pi(f_a)$ changes signs on $\RR^3$, which is a contradiction.
\end{proof}

The proof of Theorem \ref{thm:evensymmetric} is now complete.

\begin{remark} For $2\le a$
we have the decomposition
$$(\alpha x+y)^4+(\alpha x-y)^4+(\alpha y+z)^4+(\alpha y-z)^4+(\alpha z+x)^4+(\alpha z-x)^4=f_{a}$$
with $a=\frac{2\alpha ^4+2}{2\alpha ^2}$.

Note that $a=\frac{2\alpha ^4+2}{2\alpha ^2}$,  can be satisfied only for
$2\le a$.
The case $a=2$ corresponds to $\alpha=1$, and the six summands 
(and their opposite) correspond to the $12$ medium points on edges of cube (cubic decomposition).

The case $a=3$ corresponds to $\alpha=\pm\frac{1+\sqrt{5}}{2}$ (golden ratio), we have the double conic $(x^2+y^2+z^2)^2$ and the six summands 
(and their opposite) correspond to the $12$ vertices of regular icosahedron
(icosahedral decomposition), in agreement with \cite[Theor. 9.13]{Rez}.
\end{remark}

\begin{remark} In case $a=-2$, the quartic $f_{-2}$ has complex border rank $5$ (it is called Clebsch) and has the following complex
symmetric decomposition with six summands

$4f_{-2}=(x+iy)^4+(x+iz)^4+(y+iz)^4+(y+ix)^4+(z+ix)^4+(z+iy)^4$.

Note that each of the $6$ summands correspond to a line tangent to the conic $x^2+y^2+z^2$.
It can be proved that $f_{-2}$ and $f_{\infty}=x^4+y^4+z^4$ are the only even symmetric quartics admitting a symmetric decomposition.
\end{remark}

\section{Ternary Quintics}\label{quintics}

In this section we compute the typical ranks for points $f\in \mathbb{P}(S^5\mathbb{R}^3)$ with respect to the Veronese variety $X=v_5(\P 2_{\mathbb{R}})$. 

Let  $f\in S^5\RR^3$. 
As for cubic surfaces, we have uniqueness of the general decomposition
$f=\sum_{i=1}^7l_i^5$.
So there are four possibilities, all occurring for $f$
lying in some set of positive volume:

\begin{enumerate}[(i)]

\item\label{(i)} all $l_i$ are real,

\item\label{(ii)} five $l_i$ are real,

\item\label{(iii)} three $l_i$ are real,

\item\label{(iv)} one $l_i$ is real.
\end{enumerate}
In the case (\ref{(i)}) the real rank is $7$. In the case (\ref{(ii)}) the real rank is between $8$ and $10$,
because we can substitute two conjugate fifth powers with five real fifth powers. This shows the following

\begin{prop} $7$ and $8$ are typical ranks for ternary quintics.
\end{prop}

 Now, analogously as what we have done in the previous section for ternary quartics, we can show that typical ranks for ternary quintics are between 7 and 13.

\begin{prop}
Typical ranks for ternary quintics are among $7$ and $13$.
\end{prop}
\begin{proof}
For a generic $f\in S^5\mathbb{R}^3$ and a generic $Q\in\PP^2$ we have that  $P_Q(f)\in S^4\mathbb{R}^3$ and  $rk(P_Q(f))\le 8$ (by Proposition \ref{prop:typquartic}),
say $P_Q(f)=\sum_{i=1}^8l_i^3$. If $Q\notin l_i$,  then there exists $g\in S^4V$  and $c_i\in\RR$ such that 
$$f=\sum_{i=1}^8c_i l_i^4+g$$ and $P_Q(g)=0$.
Since  $g$ is a binary form  (it may be written as a polynomial in only two variables) and since for $S^5\RR^2$ we have three typical ranks $r=3,4,5$ (see \cite{CO,B})
then the typical ranks of ternary quintics are $\le 8+5=13$.

\end{proof}
 
\vskip 0.5cm

\noindent \textbf{Acknowledgement} We would like to thank F. O. Schreyer for useful talks about the paper \cite{KO}. We thank the Simons Institute for the Theory of Computing in Berkeley, CA for their generous support while in residence during the program on {\em Algorithms and Complexity in Algebraic Geometry}. The second author was partially supported by the Sloan Research Fellowship and NSF CAREER award DMS-1352073. A. Bernardi and G. Ottaviani are member of GNSAGA-INDAM.


\bigskip
{\small


\begin{thebibliography}{Dilloo Dilloo 33}


\bibitem[AH95]{AH} J. Alexander, A. Hirschowitz, Polynomial interpolation in several variables. {\it J. Alg. Geom. } 4 (1995), 201-222.

\bibitem[Ba]{Ba} M. Banchi, Rank and border rank of real ternary cubics, 
Bollettino dell'UMI, 8 (2015) 65--80. 

\bibitem[B13]{B} G. Blekherman, Typical Real Ranks of Binary Forms, {\it Foundations of Computational Math.}, 15(3),  (2015), 793--798.

\bibitem[BT14]{BT} G. Blekherman, Z. Teitler, On Maximum, Typical, and Generic Ranks, Mathematische Annalen, 362(3),  (2015), 1021-1031.

\bibitem[BCG]{BCG} M. Boji, E. Carlini, A. Geramita, Monomials as sum of powers, the real binary case, {\it Proc. Amer. Math. Soc.}, {\bf 139}, 3039--3043,
2011.

\bibitem[CCG]{CCG} E. Carlini, M.V. Catalisano, A. Geramita,  
The solution to the Waring problem for monomials and the sum of coprime monomials, {\it J. of Algebra}, {\bf 370} ,   5--14, 2012.

\bibitem[CR]{CR} A. Causa, R. Re, { On the maximum rank of a real binary form}.
{\it Annali di Matematica Pura ed Applicata}, {\bf 190} (1),  55--59, 2011. 

\bibitem[CB08]{ComoT08:lasvegas} P. Comon and J. Ten Berge, Generic and typical ranks of the three-way arrays. In {\it Icassp'08}, pages 3313--3316. Las Vegas, March 30 - April 4 2008. hal-00327627.

\bibitem[CBDC09]{ComoTDC09:laa} P. Comon, J. M. F. Ten Berge, L. DeLathauwer, and J. Castaing, Generic and typical ranks of multi-way arrays. {\it Linear Algebra Appl.}, 430(11--12):2997--3007, June 2009. hal-00410058.

\bibitem[CO09]{CO} P. Comon, G. Ottaviani, On the typical rank of real binary forms,
{\it Linear and Multilinear Algebra}, {\bf 60} (6) , 657--667 , 2012.

\bibitem[DeP]{DeP} A. De Paris, A proof that the maximal rank for plane quartics is seven, arXiv:1309.6475.

\bibitem[DoKa]{DoKa} I. Dolgachev, V. Kanev, Polar covariants of plane cubics and quartics, {\it Advances in Math.} {\bf 98} , 216--301, 1993.

\bibitem[EC85a]{MR966664} F. Enriques and O. Chisini, {\it Lezioni sulla teoria geometrica delle equazioni e delle funzioni algebriche. 1. Vol. I, II}, volume 5 of {\it Collana di Matematica [Mathematics Collection]}. Nicola Zanichelli Editore S.p.A., Bologna, 1985. Reprint of the 1924 and 1934 editions.

\bibitem[Fr]{Fr} S. Friedland,  On the generic rank of 3-tensors,
{\it Linear Algebra Appl.}, 436 (2012), 478--497.

\bibitem[Kru89]{MR1088949} J. B. Kruskal, Rank, decomposition, and uniqueness for 3-way and $N$-way arrays. In {\it Multiway data analysis (Rome , 1988)}, pages 7--18. North-Holland, Amsterdam, 1989.


\bibitem[KS04]{KO} J. Koll\`ar and F. O. Schreyer, Real Fano 3-folds of type $V_{22}$. In {\it The Fano Conference}, pages 515--531, Univ. Torino, Turin, 2004.

\bibitem[LR]{LR} G. Lachaud, C.Ritzenthaler, On some questions of Serre on abelian threefolds. Algebraic geometry and its applications, 88–115, Ser. Number Theory Appl., {\bf 5}, World Sci. Publ., Hackensack, NJ, 2008.

\bibitem[Muk92]{M92} S. Mukai,  Fano {$3$}-folds,
In {\it Complex projective geometry (Trieste, 1989/Bergen, 1989)}, volume 179 of {\it London Math. Soc. Lecture Note Ser.}, pages 255--263. Cambridge Univ. Press, Cambridge,1992.

\bibitem[Ott09]{O} G. Ottaviani, An invariant regarding Waring's problem for cubic polynomials, {\it Nagoya Math. J.}, 193:95--110, 2009.

\bibitem[RS00]{RS} K. Ranestad and F. O. Schreyer, Varieties of sums of powers, {\it J. Reine. Angew. Math.}, 525:147--181, 2000.

\bibitem[Rez]{Rez} B. Reznick, Sums of Even Powers of Real Linear Forms,
{\it Memoirs of the AMS}, {\bf 96}, n. 463, 1992.

\bibitem[Rez10]{Rez1} B. Reznick, Laws of inertia in higher degree binary forms, {\it Proc. Amer. Math. Soc.} \textbf{138}:815--826, 2010.

\bibitem[SSM]{SSM} T. Sumi, M. Miyazaki, T. Sakata, Typical ranks for $m\times n\times (m-1)n$ tensors with $m\le n$, {\it Linear Algebra and its Appl.},
{\bf 438}, 953--958, 2013.

\bibitem[tB00]{MR1818596} Jos M. F. ten Berge, The typical rank of tall three-way arrays. {\it Psychometrika}, 65(4):525--532, 2000.

\bibitem[TBK99]{MR1693919}  Jos M. F. Ten Berge and Henk A. L. Kiers, Simplicity of core arrays in three-way principal component analysis and the typical rank of $p\times q \times 2$ arrays. {\it Linear Algebra Appl.}, 294(1-3): 169--179, 1999.

\bibitem[tBS06]{MR2257591} Jos M. F. ten Berge and Alwin Stegeman, Symmetry transformations for square sliced three-way arrays, with applications to their typical rank. {\it Linear Algebra Appl.}, 418(1):215--224, 2006.

\end{thebibliography}}
\end{document}